\documentclass[11pt]{article}
\usepackage{amsfonts}
\usepackage{color}
\usepackage{amsmath,amssymb,comment}
\newtheorem{theo}{Theorem}[section]
\newtheorem{lem}[theo]{Lemma}

\newtheorem{prop}[theo]{Proposition}

\newcommand{\mysection}[1]{\section{#1} \setcounter{equation}{0}}
\newcommand{\proof}{{\sc Proof.} \quad}
\newcommand{\proofc}{{\sc Proof} \ }
\newcommand{\be}{\begin{equation} \label}
\newcommand{\ee}{\end{equation}}
\newcommand{\bea}{\begin{eqnarray}\label}
\newcommand{\eea}{\end{eqnarray}}
\newcommand{\bas}{\begin{eqnarray*}}
\newcommand{\eas}{\end{eqnarray*}}
\newcommand{\bit}{\begin{itemize}}
\newcommand{\eit}{\end{itemize}}
\newcommand{\qed}{\hfill$\Box$ \vskip.2cm}
\newcommand{\nn}{\nonumber}
\newcommand{\R}{\mathbb{R}}
\newcommand{\N}{\mathbb{N}}
\newcommand{\pO}{\partial\Omega}

\newcommand{\io}{\int_\Omega}
\newcommand{\na}{\nabla}
\newcommand{\Del}{\Delta}
\newcommand{\del}{\delta}
\newcommand{\al}{\alpha}

\newcommand{\Lam}{\Lambda}
\newcommand{\sig}{\sigma}
\newcommand{\pa}{\partial}
\newcommand{\bom}{\overline{\Omega}}
\newcommand{\Om}{\Omega}

\newcommand{\un}{\underline}

\newcommand{\vs}{\vspace*}
\newcommand{\hs}{\hspace*}

\newcommand{\vp}{\varphi}
\newcommand{\lbal}{\left\{ \begin{array}{l}}
\newcommand{\lball}{\left\{ \begin{array}{ll}}
\newcommand{\ear}{\end{array} \right.}

\newcommand{\abs}{\\[5pt]}

\newcommand{\adb}{\allowdisplaybreaks}
\newcommand{\tm}{T_{max}}

\begin{document}
\adb
\title{Hotspot formation driven by temperature-dependent coefficients\\
in one-dimensional thermoviscoelasticity}
\author{
Michael Winkler\footnote{michael.winkler@math.uni-paderborn.de}\\
{\small Universit\"at Paderborn, Institut f\"ur Mathematik}\\
{\small 33098 Paderborn, Germany} }
\date{}
\maketitle
\begin{abstract}
\noindent 
This manuscript is concerned with a two-component evolution system generalizing the classical model
for one-dimensional thermoviscoelastic dynamics in Kelvin-Voigt materials in the presence of temperature-dependent viscosities 
and elastic stiffnesses. 
Under suitable assumptions on the growth of these ingredients and on the initial data, 
the occurrence of finite-time blow-up with respect to the $L^\infty$ norm in the temperature variable is discovered.\abs
\noindent {\bf Key words:} nonlinear acoustics; thermoviscoelasticity; viscous wave equation; blow-up\\
{\bf MSC 2020:} 74H35 (primary); 35B44, 74F05, 35L05 (secondary)
\end{abstract}
%
%
%
%
%
%
%
%
%
%
\newpage
\section{Introduction}\label{intro}
Evolution equations modeling thermoviscoelasticity have thoroughly been studied in the mathematical literature.
With regard to typical application frameworks in which temperature dependencies of 
elastic parameters and viscosities 
can be neglected, far-reaching theories of well-posedness could be established.
Classical results addressing one-dimensional scenarios
assert global solvability within large ranges of parameter functions and initial data 
(\cite{racke_zheng}, \cite{dafermos}, \cite{dafermos_hsiao_smooth}, \cite{guo_zhu}, \cite{watson})
while also in multi-dimensional settings, obstacles linked to the complex interplay of vector-valued displacements
could successfully be coped with in various particular cases of physical relevance 
(\cite{roubicek}, \cite{mielke_roubicek}, \cite{blanchard_guibe}, \cite{rossi_roubicek_interfaces13}, 
\cite{gawinecki_zajaczkowski_cpaa}, 
\cite{roubicek_nodea2013},
\cite{pawlow_zajaczkowski_cpaa17}, \cite{owczarek_wielgos}).
In some one-dimensional problems, even large-time relaxation in the flavor of stabilization toward homogeneous states has been 
shown to exhaustively determine solution behavior (\cite{racke_zheng}; cf.~also \cite{bies_cieslak}, \cite{bies_cieslak_SIMA} 
and \cite{cieslak_muha_trifunovic} for corresponding developments concerned with related viscosity-free systems).\abs
Now as recently discovered, core determinants for thermoviscoelastic dynamics may well exhibit some 
moderate but non-negligible dependencies on temperature; in fact, experimental observations 
suggest that in certain piezoceramic materials,
both the elastic stiffness parameters and the viscosities may increase with temperature
(\cite{friesen}; cf.~also \cite{CLW}, \cite{Gubinyi2007}).
For simplicity ignoring further dependencies such as on strains, and moreover assuming constancy 
of a retardation time parameter $\tau$ which in the context of standard Kelvin-Voigt modeling quantifies mechanical losses,
upon conveniently normalizing variables we are thereby led to considering the evolution system
\be{00}
	\lbal
	u_{tt} = \big(\gamma(\Theta) u_{xt}\big)_x + a \big( \gamma(\Theta) u_x\big)_x - \big(f(\Theta)\big)_x, \\[1mm]
	\Theta_t = \Theta_{xx} + \gamma(\Theta) u_{xt}^2 - f(\Theta) u_{xt},
	\ear
\ee
for the unknown displacement $u=u(x,t)$ and the temperature $\Theta=\Theta(x,t)$,
where $f(\Theta)$ measures thermal dilation and hence is explicitly to $\Theta$ as usual, 
but where now moreover also the material viscosity $\gamma(\Theta)$ and the elastic stiffness $a\gamma(\Theta)$
depend on $\Theta$, while $a=\frac{1}{\tau}$ is constant (see, e.g., \cite{CLW}).\abs
The knowledge on possible effects of such deviations from classical settings with $\gamma\equiv const.$ seems yet at a rather
early stage, and appears to be restricted to results on local-in-time well-posedness in the presence of fairly
arbitrary ingredients and initial data (\cite{fricke}, \cite{claes_win}, \cite{win_AMOP}), 
on solvability up to an arbitrary but fixed finite
time horizon under smallness conditions inter alia on $\gamma'$ and $f'$ (\cite{meyer}), 
and on global existence and boundedness of solutions in the presence of small parameters $a$, small derivatives $f'$, and small
initial data (\cite{fricke}, \cite{CLW}).
In fact, this may be viewed as reflecting reduced accessibility to classical approaches to one-dimensional thermoviscoelasticity
which have crucially relied on constancy of viscosities 
(\cite{racke_zheng}, \cite{zheng_shen}, \cite{shibata}, \cite{kim}, \cite{guo_zhu}, \cite{shen_zheng_zhu}, \cite{hsiao_luo}),
with few exceptions admitting suitably mild dependencies
on the strain $u_x$ (\cite{watson}, \cite{jiang_QAM1993}).
In the presence of non-constant $\gamma$, (\ref{00}) also appears to significantly differ from its relative in which at least
the elastic parameter is assumed to be a fixed number $b$;
Indeed, while existence results for this latter system can be derived by suitably exploiting 
a global structural property expressed in the conservation law
\bas
	\frac{d}{dt} \bigg\{
	\frac{1}{2} \io u_t^2
	+ \frac{b}{2} \io u_x^2
	+ \io \Theta \bigg\}
	= 0,
\eas
formally valid actually without any restrictions on the particular behavior of $\gamma=\gamma(\Theta)$ 
(\cite{win_ZAMP}, \cite{win_JEE}), no comparable feature seems to persist when passing on to (\ref{00}).\abs
{\bf Main results.} \quad
The present manuscript intends to indicate that consequences implied by this lack of favorable structures in (\ref{00}) 
may not be of merely methodological relevance, and to reveal that, in fact, temperature dependencies of viscosities
and elastic stiffnesses may facilitate hotspot formation in the sense of a spontaneous emergence of 
singular temperature distributions within finite time.
This will be addressed in the framework of the initial-boundary value problem 
\be{0}
	\lball
	u_{tt} = \nabla \cdot (\gamma(\Theta) \na u_t) + a \na \cdot (\gamma(\Theta) \na u) + \na\cdot f(\Theta), 
	\qquad & x\in\Om, \ t>0, \\[1mm]
	\Theta_t = D\Del\Theta + \gamma(\Theta) |\na u_t|^2 + f(\Theta)\cdot\na u_t,
	\qquad & x\in\Om, \ t>0, \\[1mm]
	u=0, \quad \frac{\pa\Theta}{\pa\nu}=0,
	\qquad & x\in\pO, \ t>0, \\[1mm]
	u(x,0)=u_0(x), \quad u_t(x,0)=u_{0t}(x), \quad \Theta(x,0)=\Theta_0(x),
	\qquad & x\in\Om,
	\ear
\ee
considered here for the real-valued unknowns $u$ and $\Theta$ and
in smoothly bounded $n$-dimensional domains $\Om$.
While for $n=1$ this immediately aligns with our original ambition to describe (\ref{00}), 
the inclusion of arbitrary $n\ge 1$, which will nowhere complexify our analysis, might hint on some independence
of the phenomenon in question from spatial dimensionality.\abs
To create an appropriate setup of suitably smooth solutions, as a starting point let us state 
the following basic result on local classical solvability and extensibility that can be proved by an almost verbatim
copy of a reasoning detailed in \cite{claes_win} for the relative of (\ref{0}) involving no-flux boundary conditions for
both solutions components:
\begin{prop}\label{prop_loc}
  Let $n\ge 1$ and $\Om\subset\R^n$ be a bounded domain with smooth boundary, let $a>0$ and $D>0$, and suppose that
  \be{gf}
	\lbal
	\mbox{$\gamma\in C^2([0,\infty))$ is such that $\gamma>0$ on $[0,\infty)$, and that} \\[1mm]
  	\mbox{$f\in C^2([0,\infty);\R^n)$ satisfies $f(0)=0$.}
	\ear
  \ee
  Then whenever
  \be{Init}
	\lbal
	u_0\in C^2(\bom)
	\mbox{ is such that $u=0$ on $\pO$,} \\[1mm]
	u_{0t}\in \bigcup_{\beta\in (0,1)} C^{1+\beta}(\bom)	
	\mbox{ is such that $u_{0t}=0$ on $\pO$, \qquad and} \\[1mm]
	\Theta_0\in \bigcup_{\beta\in (0,1)} C^{1+\beta}(\bom)
	\mbox{ satisfies $\Theta_0\ge 0$ in $\Om$ and $\frac{\pa \Theta_0}{\pa\nu}=0$ on $\pO$,}
	\ear
  \ee
  one can find $\tm\in (0,\infty]$ as well as
  \be{tl1}
	\lbal
	u\in C^{1,0}(\bom\times [0,\tm)) \cap C^{2,1}(\bom\times (0,\tm))
		\qquad \mbox{and} \\[1mm]
	\Theta\in C^{1,0}(\bom\times [0,\tm)) \cap C^{2,1}(\bom\times (0,\tm))
		\qquad \mbox{with} \\[1mm]
	u_t\in C^{1,0}(\bom\times [0,\tm)) \cap C^{2,1}(\bom\times (0,\tm))
	\ear
  \ee
  such that $\Theta\ge 0$ in $\Om\times (0,\tm)$, that $(u,\Theta)$ forms a classical solution of (\ref{0})
  in $\Om\times (0,\tm)$, and that
  \bea{Ext}
	& & \hs{-15mm}
	\mbox{if $\tm<\infty$, \quad then \quad} \nn\\
	& & \hs{-6mm}
	\limsup_{t\nearrow\tm} \Big\{ \|u_t(\cdot,t)\|_{W^{1,p}(\Om)} + \|\Theta(\cdot,t)\|_{L^\infty(\Om)} \Big\}
		=\infty
	\mbox{\quad for all $p\ge 2$ such that $p>n$.}
  \eea
\end{prop}

\vs{4mm}
{\bf A refined extensibility criterion: Necessity of singular $\Theta$ at finite-time blow-up.} \quad
The first of our specific goals now is to make sure that any finite-time singularity formation in (\ref{0}) must actually 
go along with a phenomenon of immediate physical relevance:
Indeed, (\ref{Ext}) can be refined in such a way that its revised version actually asserts unboundedness with respect to the norm
$L^\infty$ already of the temperature variable alone, hence indicating that some genuine hotspot formation must
occur in any such event.
This difference from classical results on $C^2$-blow-up in one- and higher-dimensional models for
thermoelasticity (\cite{racke_bu}, \cite{dafermos_hsiao}, \cite{hrusa}) may be regarded as a consequence of the mere presence
of viscosity. 
In fact, by strongly relying on corresponding dissipative effects in the course of variational
arguments firstly at zero-order level in the framework of a Moser iteration in Section \ref{sect2}, 
and then at first order in Section \ref{sect3}, we can sharpen (\ref{Ext}) as follows.
\begin{prop}\label{prop_ext}
  Suppose that $n\ge 1$ and $\Om\subset\R^n$ is a bounded domain with smooth boundary, that $a>0$ and $D>0$, 
  and that (\ref{gf}) and (\ref{Init}) hold.
  Then the quantities $\tm,u$ and $\Theta$ obtained in Proposition \ref{prop_loc} are such that, actually,
  \be{ext}
	\mbox{if $\tm<\infty$, \quad then \quad} 
	\limsup_{t\nearrow\tm} \|\Theta(\cdot,t)\|_{L^\infty(\Om)} 
	= \infty.
  \ee
\end{prop}

\vs{4mm}
{\bf Occurrence of finite-time blow-up for superlinearly growing $\gamma$.} \quad
In Section \ref{sect4}, the core of our analysis will thereafter be devoted to the detection of blow-up under suitable
assumptions mainly on the crucial ingredient $\gamma$.
In the context of an argument centered around an inequality of the form
\be{001}
	\frac{d}{dt} \io \psi(\Theta)
	+ \frac{1}{2} \io |\na u_t|^2
	\le C,
	\quad t\in (0,T),
	\qquad 
	\psi(\xi)= - \int^\xi \frac{d\sig}{\gamma(\sig)},
	\quad \xi> 0,
\ee
with some $C>0$ depending on $f,\gamma$ and $T>0$ in a suitably cotrollable manner (Lemma \ref{lem601}),
a key challenge to be overcome will consist in adequately controlling the favorably signed but potentially oscillatory
contribution $|\na u_t|^2$ to the second equation in (\ref{0}) from below.
Based on the simple observation that
\be{01}
	\io |\na u_0|^2
	\le 2 \io |\na u(\cdot,t)|^2 + 2t\int_0^t \io |\na u_t|^2
\ee
(Lemma \ref{lem63}), this will be achieved 
by making substantial use of the particular link between viscosity and elastic parameters in (\ref{0}),
which will, namely, enable us to derive 
an upper bound for 
\bas
	\int_0^T \io |\na u_t + a\na u|^2 
\eas
through an appropriate testing procedure
(Lemma \ref{lem61}).
Due to the upper bound for $\int_0^T \io |\na u|^2$ in terms of $\int_0^T \io |\na u_t|^2$ thereby implied, within suitably 
small time intervals this can be used together with (\ref{01}) to indeed confirm that with respect to the occurrence of blow-up,
the heat evolution subsystem of (\ref{0})
shares essential features of the corresponding ODEs
\be{02}
	\theta_t = c \gamma(\theta)
\ee
with $c>0$.\abs
Suitably tracing dependencies in the respectively obtained inequalities 
especially on all components of the initial data will draw two different consequences
on explosion-enforcing constellations in (\ref{0}).
In fact, the first of our main results in this direction will show that for any nondecreasing function $\gamma$ which
is such that positive solutions of (\ref{02}) blow up, 
and within large classes of $f$,
sufficiently strong initial strains enforce the emergence of hotspots in finite time.
We emphasize that the following result in this regard does not rely on any 
assumption on $\Theta_0$ beyond the mere hypothesis in (\ref{Init}),
and we note that the requirements on $\gamma$ and $f$ in (\ref{gp}), (\ref{gi}) and (\ref{gf1}) are particularly satisfied if, e.g., 
\bas
	\gamma(\xi)=\gamma_\star \cdot (\xi+1)^\mu
	\quad \mbox{and} \quad
	f(\xi)=\xi f_\star 
	\qquad \mbox{for all } \xi\ge 0
\eas
with some 
\bas
	\mu \ge 2
\eas
and some $\gamma_\star>0$ and $f_\star \in \R^n$:
\begin{theo}\label{theo65}
  Let $n\ge 1$ and $\Om\subset\R^n$ be a bounded domain with smooth boundary, let $a>0$, assume that (\ref{gf}) holds,
  and suppose that moreover
  \be{gp}
	\gamma'(\xi)\ge 0
	\qquad \mbox{for all } \xi>0
  \ee
  and
  \be{gi}
	\int_0^\infty \frac{d\xi}{\gamma(\xi)} < \infty,
  \ee
  and that there exists $\Lam>0$ such that
  \be{gf1}
	\frac{|f|^2(\xi)}{\gamma(\xi)} \le \Lam
	\qquad \mbox{for all } \xi\ge 0.
  \ee
  Then for all $T>0$ there exists $C=C(a,\gamma,f,T)>0$ such that whenever $D>0$ and (\ref{Init}) holds with
  \be{65.1}
	\io |\na u_0|^2 \ge C \io u_0^2 + C \io u_{0t}^2 + C,
  \ee
  the corresponding solution of (\ref{0}) blows up before or at time $T$: In Proposition \ref{prop_loc}, we then have
  $\tm \le T$ and 
  \be{65.2}
	\limsup_{t\nearrow\tm} \|\Theta(\cdot,t)\|_{L^\infty(\Om)} = \infty.
  \ee
\end{theo}
Alternatively, however, also arbitrarily small initial distributions $u_0$ and $u_{0t}$ may lead to blow-up,
provided that the temperature is suitably large throughout the domain at the initial instant:
\begin{theo}\label{theo66}
  Let $n\ge 1$ and $\Om\subset\R^n$ be a smoothly bounded domain, and suppose that $a>0$,
  and that $\gamma$ and $f$ satisfy (\ref{gf}), (\ref{gp}), (\ref{gi}) and (\ref{gf1}) with some $\Lam>0$.
  Then given any $T>0, \eta>0$ and $M>0$, one can find $C(a,\gamma,f,T,\eta,M)>0$ with the property that if $D>0$, and if
  (\ref{Init}) is satisfied with
  \be{66.1}
	\io |\na u_0|^2 \ge \eta
  \ee
  and
  \be{66.2}
	\io u_0^2 + \io u_{0t}^2 \le M,
  \ee
  then assuming that
  \be{66.3}
	\Theta_0(x) \ge C(a,\gamma,f,T,\eta,M)
	\qquad \mbox{for all } x\in\Om
  \ee
  ensures that for the solution of (\ref{0}) from Proposition \ref{prop_loc} we have $\tm\le T$, and that (\ref{65.2}) holds.
\end{theo}
{\bf Remark.} \quad
i) \ Our argument will actually show that whenever $T>0$, blow-up before or at time $T>0$ occurs if $a>0$ and $D>0$, if
$\gamma$ and $f$ satisfy (\ref{gf}), (\ref{gp}) and (\ref{gi}), and if $(u_0,u_{0t},\Theta_0)$ complies with (\ref{Init})
and is such that
\bas
	\io |\na u_0|^2 > \Big\{ \frac{4}{a^2 T} + 2T\Big\} \cdot \io \psi(\Theta_0) 
	+ \frac{4 e^{2aT}}{a^2 T} \cdot \Big\{ \gamma \Big(\inf_{x\in\Om} \Theta_0(x) \Big) \Big\}^{-1} 
		\cdot \bigg\{ \frac{3a^2}{2} \io u_0^2 + \io u_{0t}^2 \bigg\},
\eas
where $\psi(\xi):=\int_\xi^\infty \frac{d\sig}{\gamma(\sig)}$, $\xi\ge 0$
(cf.~Lemma \ref{lem64} below).\abs
ii) \ Through the design described above, our approach toward blow-up detection in (\ref{0}) can be viewed as a hybrid reasoning
that joins elements from parabolic blow-up analysis with arguments adapted to the particular structure of the viscous wave
equation in (\ref{0}).
This attempts to adequately take into account the circumstance that (\ref{0}) apparently is not accessible to 
classical nonexistence proofs drawing either explicitly or implicitly on the presence of genuine energies unbounded from below,
as seen in numerous discoveries of explosions in superlinearly forced wave equations
(\cite{levine}, \cite{gazzola}, \cite{alves}), and also in various viscoelastic models for either solid or fluid materials
(\cite{dehua_wang}, \cite{messaoudi}, \cite{boudjeriou}, \cite{nhan}, \cite{song}, \cite{sun}).
\mysection{The role of $\Theta$ near possible singularities. Proof of Proposition \ref{prop_ext}}
\subsection{Local-in-time $L^\infty$ bounds for $u$ and $u_t$}\label{sect2}
To simplify our presentation in this and the next sections, let us note that letting $v:=u_t+au$ transfers
(\ref{0}) into the three-component system
\be{0v}
	\lball
	v_t = \na\cdot (\gamma(\Theta)\na v) + av - a^2 u + \na\cdot f(\Theta),
	\qquad & x\in\Om, \ t>0, \\[1mm]
	u_t = v-au,
	\qquad & x\in\Om, \ t>0, \\[1mm]
	\Theta_t = D\Del\Theta + \gamma(\Theta) |\na v - a\na u|^2 + f(\Theta)\cdot (\na v - a \na u),
	\qquad & x\in\Om, \ t>0, \\[1mm]
	v=0, \quad u=0, \quad \frac{\pa\Theta}{\pa\nu}=0,
	\qquad & x\in\pO, \ t>0, \\[1mm]
	v(x,0)=v_0(x):=u_{0t}(x)+au_0(x), \quad u(x,0)=u_0(x), \quad \Theta(x,0)=\Theta_0(x),
	\qquad & x\in\Om,
	\ear
\ee
and our first objective will consist in showing that if (\ref{ext}) was violated for some solution of (\ref{0}),
then in the accordingly transformed variables, both $v$ and $u$ would belong to $L^\infty(\Om\times (0,\tm))$.
Through a H\"older bound and hence a uniform continuity property 
thereby implied (Lemma \ref{lem11}), the former of these inclusions will play a key role 
in the course of an Ehrling-type interpolation argument on which our proof of Proposition \ref{prop_ext} will
rely (see Lemma \ref{lem_LW}).\abs
To achieve such $L^\infty$ bounds by means of a Moser-type iteration drawing on parabolicity of the first equation in (\ref{0v}),
given $a>0$ we fix any $\kappa>0$ large enough such that
\be{A}
	A:=a+\kappa
	\qquad \mbox{satisfies} \qquad
	A \ge 2a 
	\quad \mbox{and} \quad
	A \ge 2,
\ee
and assuming (\ref{gf}) and (\ref{Init}) we let $\tm$ and $(v,u,\Theta)$ be as above and set
\be{wz}
	w(x,t):=e^{-\kappa t} v(x,t)
	\quad \mbox{and} \quad
	z(x,t):=e^{-\kappa t} u(x,t),
	\qquad x\in\bom, \ t\in [0,\tm),
\ee
noting that then, by (\ref{0v}),
\be{0w}
	\lball
	w_t = \na\cdot (\gamma(\Theta)\na w) - (A-2a) w - a^2 z + e^{-\kappa t} \na \cdot f(\Theta),
	\qquad & x\in\Om, \ t\in (0,\tm), \\[1mm]
	z_t = w-Az,
	\qquad & x\in\Om, \ t\in (0,\tm), \\[1mm]
	w=0,
	\qquad & x\in\pO, \ t\in (0,\tm), \\[1mm]
	w(x,0)=u_{0t}(x)+au_0(x), \quad z(x,0)=z_0(x), 
	\qquad & x\in\Om.
	\ear
\ee
The following observation forms the basis for our recursive argument, using that (\ref{A}) induces sufficiently strong
zero-order dissipation in (\ref{0w}) to create some favorable absorptive contribution to the evolution of
$t\mapsto \io w^p + \io z^p$ for even $p\ge 2$:
\begin{lem}\label{lem91}
  Let $a>0$ and $D>0$, assume (\ref{gf}) and (\ref{Init}), and suppose that $\tm<\infty$, but that
  \be{91.1}
	\limsup_{t\nearrow \tm} \|\Theta(\cdot,t)\|_{L^\infty(\Om)} < \infty.
  \ee
  Then there exists $C>0$ such that for any even integer $p\ge 2$, the functions in (\ref{wz}) satisfy
  \bea{91.2}
	& & \hs{-20mm}
	\frac{d}{dt} \bigg\{ 1 + \io w^p + \io z^p \bigg\}
	+ \frac{1}{C} \cdot \bigg\{ 1 + \io w^p + \io z^p \bigg\}
	+ \frac{1}{C} \io |\na w^\frac{p}{2}|^2 \nn\\
	&\le& C p^2 \io w^p + C p^2
	\qquad \mbox{for all } t\in (0,\tm).
  \eea
\end{lem}
\proof
  If $\tm<\infty$ but (\ref{91.1}) was valid, then by continuity of $\gamma$ and $f$ and by positivity of $\gamma$, we could find 
  $c_1>0$ and $c_2>0$ such that
  \be{91.3}
	\gamma(\Theta)\ge c_1
	\quad \mbox{and} \quad
	|f(\Theta)| \le c_2
	\qquad \mbox{in } \Om\times (0,\tm),
  \ee
  and for even integers $p\ge 2$, we can integrate by parts in the first equation from (\ref{0w}) to see that, by (\ref{91.3}),
  \bea{91.5}
	\frac{d}{dt} \io w^p
	&=& p \io w^{p-1} \na\cdot\big(\gamma(\Theta)\na w\big)
	- (A-2a) p \io w^p
	- a^2 p \io w^{p-1} z
	+ p e^{-\kappa t} \io w^{p-1} \na\cdot f \nn\\
	&=& -p(p-1) \io \gamma(\Theta) w^{p-2} |\na w|^2
	- (A-2a)p \io w^p
	- a^2 p \io w^{p-1} z \nn\\
	& & - p(p-1) e^{-\kappa t} \io w^{p-2} \na w\cdot f
	+ p e^{-\kappa t} \int_{\pO} w^{p-1} f\cdot\nu \nn\\
	&\le& - c_1 p(p-1) \io w^{p-2} |\na w|^2
	- a^2 p \io w^{p-1} z \nn\\
	& & + c_2 p(p-1) \io w^{p-2} |\na w|
	+ c_2 p \int_{\pO} |w^{p-1}|
	\qquad \mbox{for all } t\in (0,\tm),
  \eea
  because $A-2a\ge 0$ by (\ref{A}).
  Here, Young's inequality shows that
  \bas
	c_2 p(p-1) \io w^{p-2} |\na w| 
	&\le& \frac{c_1 p(p-1)}{2} \io w^{p-2} |\na w|^2
	+ \frac{c_2^2 p(p-1)}{2c_1} \io w^{p-2} \nn\\
	&\le& \frac{c_1 p(p-1)}{2} \io w^{p-2} |\na w|^2
	+ c_3 p^2 \io w^p + c_3 p^2
	\qquad \mbox{for all } t\in (0,\tm)
  \eas
  with $c_3:=\frac{c_2^2}{2c_1}(1+|\Om|)$.
  As furthermore, again by Young's inequality, and due to the rough estimates $(\frac{4}{A})^\frac{1}{p-1} \le 2$
  and $a^\frac{2p}{p-1} \le a^4$ resulting from the inequalities $A\ge 2$ and $p\ge 2$,
  \bas
	- a^2 p \io w^{p-1} z
	&=& - \io \Big\{ \frac{Ap}{4} z^p \Big\}^\frac{1}{p} \cdot \Big(\frac{4}{Ap}\Big\}^\frac{1}{p} a^2 p w^{p-1} \\
	&\le& \frac{Ap}{4} \io z^p + \Big( \frac{4}{A}\Big)^\frac{1}{p-1} a^\frac{2p}{p-1} p \io w^p \\
	&\le& \frac{Ap}{4} \io z^p + c_4 p \io w^p
	\qquad \mbox{for all } t\in (0,\tm)
  \eas
  with $c_4:= 2a^4$, from (\ref{91.5}) we infer that
  \be{91.6}
	\frac{d}{dt} \io w^p
	+ \frac{c_1 p(p-1)}{2} \io w^{p-2} |\na w|^2
	\le (c_3+c_4) p^2 \io w^p
	+ \frac{Ap}{4} \io z^p + c_3 p^2
	\qquad \mbox{for all } t\in (0,\tm),
  \ee
  because $p\le p^2$.
  Apart from that, using the second equation in (\ref{0w}) along with Young's inequality we find that since $A\ge 2$ by (\ref{A}),
  \bas
	\frac{d}{dt} \io z^p
	&=& p \io w z^{p-1} 
	- Ap \io z^p \\
	&=& \io \Big\{ \frac{Ap}{2} z^p \Big\}^\frac{p-1}{p} \cdot \Big(\frac{2}{Ap}\Big)^\frac{p-1}{p} pw
	- Ap \io z^p \\
	&\le& \Big(\frac{2}{A}\Big)^{p-1} \io w^p
	- \frac{Ap}{2} \io z^p \\
	&\le& p^2 \io w^p
	- \frac{Ap}{2} \io z^p 
	\qquad \mbox{for all } t\in (0,\tm),
  \eas
  so that from (\ref{91.6}) it follows that
  \bea{91.55}
	& & \hs{-20mm}
	\frac{d}{dt} \bigg\{ 1  + \io w^p + \io z^p \bigg\}
	+ \frac{Ap}{4} \cdot \bigg\{ 1  + \io w^p + \io z^p \bigg\}
	+ \frac{c_1 p(p-1)}{2} \io w^{p-2} |\na w|^2 \nn\\
	&\le& (c_3+c_4) p^2 \io w^p
	+ p^2 w^p 
	+ \frac{Ap}{4} \io w^p
	+ c_3 p^2
	+ \frac{Ap}{4} \nn\\
	&\le& \Big( c_3 + c_4 + 1 + \frac{A}{4}\Big) p^2 \io w^p
	+ \Big( c_3 + \frac{A}{4}\Big) p^2
	+ c_2 p \int_{\pO} |w^{p-1}|
	\qquad \mbox{for all } t\in (0,\tm).
  \eea
  Here, the boundary integral can be estimated on the basis of a known boundary trace embedding estimate (\cite{alt}),
  according to which we can find $c_5>0$ such that
  \bas
	\int_{\pO} |\rho| \le c_5 \io |\na\rho| + c_5 \io |\rho|
	\qquad \mbox{for all } \rho\in C^1(\bom).
  \eas
  Therefore, due to the Cauchy-Schwarz inequality and Young's inequality we have
  \bas
	c_2 p \int_{\pO} |w^{p-1}|
	&=& c_2 p \|w^{p-1}\|_{L^1(\pO)} \\
	&\le& c_2 p \cdot \bigg\{ (p-1) c_5 \io w^{p-2} |\na w| 
	+ c_5 \io |w|^{p-1} \bigg\} \\
	&\le& c_2 c_5 p(p-1) \cdot \bigg\{ \io w^{p-2} |\na w|^2m\bigg\}^\frac{1}{2} \cdot \bigg\{ \io w^{p-2} \bigg\}^\frac{1}{2}
	+ c_2 c_5 p \io |w|^{p-1} \\
	&\le& \frac{c_1 p(p-1)}{4} \io w^{p-2} |\na w|^2
	+ \frac{c_2^2 c_5^2 p(p-1)}{c_1} \io w^{p-2}
	+ c_2 c_5 p \io |w|^{p-1} \\
	&\le& \frac{c_1 p(p-1)}{4} \io w^{p-2} |\na w|^2
	+ \Big\{ \frac{c_2^2 c_5^2}{c_1} + c_2 c_5 \Big\} \cdot p^2 \cdot \io w^p
	+ c_2 c_5 p^2 |\Om|
  \eas
  for all $t\in (0,\tm)$, again because $p\le p^2$.
  Noting that
  \bas
	\frac{c_1 p(p-1)}{4} \io w^{p-2} |\na w|^2
	= \frac{c_1 (p-1)}{p} \io \big|\na w^\frac{p}{2}\big|^2
	\ge \frac{c_1}{2} \io \big|\na w^\frac{p}{2}\big|^2
	\qquad \mbox{for all } t\in (0,\tm),
  \eas
  from (\ref{91.55}) we thereby obtain (\ref{91.2}).
\qed
A first application concretizes this to the case when $p=2$, and thereby forms the starting point for our iteration:
\begin{lem}\label{lem92}
  If $a>0$ and $D>0$ and (\ref{gf}) as well as (\ref{Init}) hold, and if $\tm<\infty$ but (\ref{91.1}) is satisfied, 
  then there exists $C>0$ such that for $w$ and $z$ as in (\ref{wz}) we have
  \bea{92.1}
	\io w^2(\cdot,t) 
	+ \io z^2(\cdot,t) 
	\le C
	\qquad \mbox{for all } t\in (0,\tm).
  \eea
\end{lem}
\proof
  On dropping some favorably signed summands therein, from (\ref{91.2}) we particularly obtain $c_1>0$ such that
  \bas
	\frac{d}{dt} \bigg\{ 1 + \io w^2 + \io z^2 \bigg\}
	\le c_1 \io w^2 + c_1 
	\le c_1 \cdot \bigg\{ 1 + \io w^2 + \io z^2 \bigg\}
	\qquad \mbox{for all } t\in (0,\tm).
  \eas
  As $\tm$ is assumed to be finite, in view of Gronwall's lemma this entails (\ref{92.1}).
\qed
Next applying Lemma \ref{lem91} to exponentially growing $p$ in Moser's style, we can accomplish the main step of this Section:
\begin{lem}\label{lem93}
  Let $a>0$ and $D>0$ and (\ref{gf}) as well as (\ref{Init}) be satisfied, and suppose that $\tm<\infty$ but that (\ref{91.1}) holds.
  Then one can find $C>0$ such that the functions defined in (\ref{wz}) have the property that
  \bea{93.1}
	\|w(\cdot,t)\|_{L^\infty(\Om)}
	+ \|z(\cdot,t)\|_{L^\infty(\Om)}
	\le C
	\qquad \mbox{for all } t\in (0,\tm).
  \eea
\end{lem}
\proof
  In order to estimate the expressions
  \be{93.02}
	N_k:=\sup_{t\in (0,\tm)} \bigg\{ 1 + \io w^{p_k}(\cdot,t) + \io z^{p_k}(\cdot,t) \bigg\},	
	\qquad k\in\N=\{0,1,2,...\},
  \ee
  with $p_k:=2^{k+1}$, $k\in\N$, we first note that $N_0$ is finite according to Lemma \ref{lem92},
  and in line with a Gagliardo-Nirenberg inequality we fix $c_1>0$ in such a way that
  \be{93.2}
	\|\vp\|_{L^2(\Om)}^2
	\le c_1 \|\na\vp\|_{L^2(\Om)}^\frac{2n}{n+2} \|\vp\|_{L^1(\Om)}^\frac{4}{n+2}
	\qquad \mbox{for all $\vp\in C^1(\bom)$ such that $\vp=0$ on $\pO$.}
  \ee
  We next employ Lemma \ref{lem91} to find $c_2>0$ and $c_3>0$ such that for any choice of $k\ge 1$, the function $y_k$ defined by
  letting
  \bas
	y_k(t):=1+ \io w^{p_k}(\cdot,t) + \io z^{p_k}(\cdot,t),
	\qquad t\in [0,\tm),
  \eas
  satisfies
  \be{93.3}
	y_k'(t) + c_2 y_k(t) + c_2 \io \big|\na w^\frac{p_k}{2}\big|^2
	\le c_3 p_k^2 \io w^{p_k} + c_3 p_k^2
	\qquad \mbox{for all } t\in (0,\tm),
  \ee 
  where by (\ref{93.2}), (\ref{93.02}) and Young's inequality,
  \bas
	c_3 p_k^2 \io w^{p_k}
	&\le& c_1 c_3 p_k^2 \big\|\na w^\frac{p_k}{2}\big\|_{L^2(\Om)}^\frac{2n}{n+2} \|w^\frac{p_k}{2}\|_{L^1(\Om)}^\frac{4}{n+2} \\
	&\le& c_1 c_3 p_k^2 N_{k-1}^\frac{4}{n+1} \big\|\na w^\frac{p_k}{2}\big\|_{L^2(\Om)}^\frac{2n}{n+2} \\
	&=& \bigg\{ c_2 \io \big| \na w^\frac{p_k}{2}\big|^2 \bigg\}^\frac{n}{n+2} 
		\cdot c_1 c_2^{-\frac{n}{n+2}} c_3 p_k^2 N_{k-1}^\frac{4}{n+2} \\
	&\le& c_2 \io \big|\na w^\frac{p_k}{2}\big|^2
	+ c_4 p_k^{n+2} N_{k-1}^2
	\qquad \mbox{for all } t\in (0,\tm)
  \eas
  with $c_4:=(c_1 c_2^{-\frac{n}{n+2}} c_3)^\frac{n+2}{2}$.
  As $N_{k-1} \ge 1$ for each $k\ge 1$, from (\ref{93.3}) we thus obtain that if we let $c_5:=c_3+c_4$, then
  \bas
	y_k'(t) + c_2 y_k(t) \le c_5 p_k^{n+2} N_{k-1}^2
	\qquad \mbox{for all $t\in (0,\tm)$ and } k\ge 1,
  \eas
  and that thus, by an ODE comparison,
  \bas
	y_k(t) \le \max \Big\{ y_k(0) \, , \, \frac{c_5}{c_2} p_k^{n+2} N_{k-1}^2 \Big\}
	\qquad \mbox{for all $t\in (0,\tm)$ and } k\ge 1.
  \eas
  By an inductive argument, in line with (\ref{93.02}) this firstly ensures that $N_k<\infty$ also for all $k\ge 1$, and that,
  secondly,
  \be{93.4}
	N_k \le \max \Big\{ c_6^{2^k} \, , \, c_7^k N_{k-1}^2 \Big\}
	\qquad \mbox{for all } k\ge 1
  \ee
  with $c_6:=(|\Om|+1) \cdot \big( 1 + \|u_{0t}+au_0\|_{L^\infty(\Om)} + \|u_0\|_{L^\infty(\Om)}\big)$
  and $c_7:=(1+\frac{c_5}{c_2}) \cdot 2^{n+2}$, because by (\ref{0w}),
  \bas
	y_k(0)
	&=& 1 + \io (u_{0t} + au_0)^{p_k} + \io u_0^{p_k} \\
	&\le& 1 + |\Om| \cdot \|u_{0t} + au_0\|_{L^\infty(\Om)}^{p_k} + |\Om| \cdot \|u_0\|_{L^\infty(\Om)}^{p_k} \\
	&\le& \big(|\Om|+1\big)^{2^k} \cdot \Big\{ 1 + \|u_{0t} + au_0\|_{L^\infty(\Om)}^{2^k} + \|u_0\|_{L^\infty(\Om)}^{2^k}
		\Big\} \\
	&\le& c_6^{2^k}
	\qquad \mbox{for all } k\ge 1
  \eas
  and
  \bas
	\frac{c_5}{c_2} p_k^{n+2}
	\le\Big(1+\frac{c_5}{c_2}\Big)^k \cdot 2^{(n+2)k}
	\le c_7^k
	\qquad \mbox{for all } k\ge 1.
  \eas
  According to the outcome of an elementary recursion documented in \cite[Lemma 2.1]{ding_win1}, however, (\ref{93.4}) entails that
  \bas
	N_k^\frac{1}{2^k} \le c_7^2 \cdot \max \big\{ c_6 \, , \, N_0\big\}
	\qquad \mbox{for all } k\ge 1,
  \eas
  from which (\ref{93.1}) readily follows.
\qed
Within finite time intervals, drawing a corresponding conclusion for $(v,u)$ simply amounts to controlling
an exponentially growing and hence bounded factor:
\begin{lem}\label{lem94}
  Let $a>0$ and $D>0$, assume (\ref{gf}) and (\ref{Init}), and suppose that $\tm<\infty$ but that (\ref{91.1}) holds.
  Then there exists $C>0$ such that
  \bea{94.1}
	\|v(\cdot,t)\|_{L^\infty(\Om)}
	+ \|u(\cdot,t)\|_{L^\infty(\Om)}
	\le C
	\qquad \mbox{for all } t\in (0,\tm).
  \eea
\end{lem}
\proof
  In view of (\ref{wz}), this is a direct consequence of Lemma \ref{lem93}.
\qed
Now again thanks to the parabolic character of the first equation in (\ref{0v}), we may utilize 
standard regularity theory to turn the $L^\infty$ bounds thus obtained into H\"older estimates for $v$.
\begin{lem}\label{lem11}
  Assume that $a>0$ and $D>0$, that (\ref{gf}) and (\ref{Init}) hold, that $\tm<\infty$ and that (\ref{91.1}) holds.
  Then there exist $\al\in (0,1)$ and $C>0$ such that
  \be{11.2}
	\|v(\cdot,t)\|_{C^\al(\bom)} \le C
	\qquad \mbox{for all } t\in (0,\tm).
  \ee
\end{lem}
\proof
  We write $B_1(x,t):=\gamma\big(\Theta(x,t)\big)$, $B_2(x,t):=f(\Theta(x,t))$ and $B_3(x,t):=av(x,t)-a^2 u(x,t)$ 
  for $(x,t)\in\Om\times (0,\tm)$, and
  may then re-spell the first equation in (\ref{0v}) according to
  \bas
	v_t = \na \cdot \big(B_1(x,t)\na v\big) + \na \cdot B_2(x,t) + B_3(x,t),
	\qquad x\in\Om, \ t\in (0,\tm).
  \eas
  here, again thanks to the continuity of $\gamma$ and $f$ on $[0,\infty)$ and due to the positivity of $\gamma$ on $[0,\infty)$, 
  (\ref{91.1}) ensures the existence of
  $c_1>0$, $c_2>0$ and $c_3>0$ such that $c_1\le B_1(x,t)\le c_2$ and $|B_2(x,t)| \le c_3$ for all $(x,t)\in\Om\times (0,\tm)$,
  and since Lemma \ref{lem94} warrants boundedness of $B_3$ in $\Om\times (0,\tm)$, 
  a standard result on parabolic H\"older regularity (\cite{PV}) yields $\al\in (0,1)$ such that
  $v\in C^{\al,\frac{\al}{2}}(\bom\times [0,\tm])$.
  This entails (\ref{11.2}) as a particular consequence.
\qed
\subsection{First-oder testing procedures. A refined extensibility criterion}\label{sect3}
According to (\ref{Ext}), verifying Proposition \ref{prop_ext} has now been reduced to making sure that
if $\tm$ was finite and $\Theta$ was bounded, then both $\na v$ and $\na u$ should lie in $L^\infty((0,\tm);L^p(\Om;\R^n))$
with some $p\ge 2$ such that $p>n$. 
This will be achieved by tracking the functional
\be{04}
	\io |\na v|^p + \io |\na u|^{p+2} + \io |\na\Theta|^p
\ee
for such $p$, and it will turn out that corresponding nonlinear influences on its evolution can be controlled
by two independent interpolation arguments which with regard to $\Theta$ only use the presupposed $L^\infty$ bounds
on their lower end, while the slightly higher regularity information on $v$ provided by Lemma \ref{lem11} facilitates
some Ehrling-type property that will play a major role in our derivation of Proposition \ref{prop_ext}.\abs
Specifically, the first of these interpolation inequalities is essentially of standard Gagliardo-Nirenberg type,
interestingly involving an explicit constant independent of domain properties.
\begin{lem}\label{lem2}
  If $\vp\in C^2(\bom)$ satisfies $\vp=0$ on $\pO$, then
  \be{2.1}
	\io|\na\vp|^{p+2} 
	\le (p+\sqrt{n})^2 \|\vp\|_{L^\infty(\Om)}^2 \io |\na\vp|^{p-2} |D^2\vp|^2.
  \ee
\end{lem}
\proof
  Since $|\Del\vp|\le \sqrt{n}|D^2\vp|$, integrating by parts and using the Cauchy-Schwarz inequality shows that
  \bas
	\io |\na \vp|^{p+2}
	&=& - \io \vp \cdot \Big\{ p|\na\vp|^{p-2} \na\vp\cdot (D^2\vp\cdot\na\vp) + |\na\vp|^p \Del\vp \Big\} \\
	&\le& (p+\sqrt{n}) \|\vp\|_{L^\infty(\Om)} \io |\na\vp|^p |D^2\vp| \\
	&\le& (p+\sqrt{n}) \|\vp\|_{L^\infty(\Om)} \cdot \bigg\{ \io |\na\vp|^{p+2}\bigg\}^\frac{1}{2} \cdot
		\bigg\{ \io |\na\vp|^{p-2} |D^2\vp|^2 \bigg\}^\frac{1}{2},
  \eas
  from which (\ref{2.1}) follows.
\qed
The second inequality differs from the above one mainly due to the appearance of an arbitrarily small factor
carried by the second-order expression on the right; for a proof, we may refer to \cite[Lemma 10.1]{taowin_PLMS}:
\begin{lem}\label{lem_LW}
  Let $\omega:(0,\infty)\to (0,\infty)$
  be nondecreasing. Then given any $p\ge 2$ and $\eta>0$, one can find $C(p,\eta)>0$ with the property that for each
  $\vp\in C^2(\bom)$ fulfilling $\frac{\pa\vp}{\pa\nu}=0$ on $\pO$ and
  \bas
	|\vp(x)-\vp(y)| < \del
	\qquad \mbox{whenever $x,y\in \bom$ are such that $ |x-y|<\omega(\del)$,}
  \eas
  the inequality
  \bas
	\io |\na\vp|^{p+2} 
	\le \eta \io |\na\vp|^{p-2} |D^2\vp|^2
	+ C(p,\eta) \|\vp\|_{L^\infty(\Om)}^{p+2}
  \eas
  holds.
\end{lem}
The further preparations will be used to control boundary integrals influencing the evolution of the functional in (\ref{04}).
The first auxiliary statement in this regard recalls two known one-sided pointwise estimates for boundary derivatives
of gradients of functions which at zero-order level satisfy either homogeneous Dirichlet or Neumann boundary conditions.
Proofs can be found in \cite[Lemma 3.4]{black_win} and \cite[Lemma 4.2]{mizoguchi_souplet}.
\begin{lem}\label{lem71}
  There exists $C>0$ such that if
  $\vp\in C^2(\bom)$ is such that $\vp=0$ on $\pO$,
  \bas
	\frac{\partial |\na\vp|^2}{\pa \nu}
	\le 2\frac{\pa \vp}{\pa\nu} \Del \vp
	+ C \Big| \frac{\pa\vp}{\pa\nu}\Big|^2
	\qquad \mbox{on } \pO,
  \eas
  and that whenever $\vp\in C^2(\bom)$ satisfies $\frac{\pa\vp}{\pa\nu}=0$ on $\pO$, we have
  \bas
	\frac{\partial |\na\vp|^2}{\pa \nu}
	\le 
	C \Big| \frac{\pa\vp}{\pa\nu}\Big|^2
	\qquad \mbox{on } \pO.
  \eas
\end{lem}
A second preliminary records a consequence entailed by continuity of a boundary trace embedding.
\begin{lem}\label{lem72}
  Let $p\ge 2$. Then there exists $C(p)>0$ such that for each $\vp\in C^2(\bom)$,
  \be{72.1}
	\int_{\pO} |\na\vp|^{p+1}
	\le C(p) \cdot \bigg\{ \io |\na\vp|^{p-2} |D^2\vp|^2 + \io |\na\vp|^p \bigg\}^\frac{1}{2}
		\cdot \bigg\{ \io |\na\vp|^{p+2} \bigg\}^\frac{1}{2}.
  \ee
\end{lem}
\proof
  Again drawing on continuity of the boundary trace embedding from $W^{1,1}(\Om)$ into $L^1(\pO)$ (\cite{alt}), we fix
  $c_1>0$ such that
  \bas
	\int_{\pO} |\rho| \le c_1 \io |\na\rho| + c_1 \io |\rho|
	\qquad \mbox{for all } \rho\in C^1(\bom),
  \eas
  which when applied to $\rho:=|\na\vp|^{p+1}$ for $p\ge 2$ and $\vp\in C^2(\bom)$ shows that
  \bas
	\int_{\pO} |\na\vp|^{p+1}
	\le \frac{p+1}{2} c_1 \io |\na\vp|^{p-1} \big| \na |\na\vp|^2 \big| + c_1 \io |\na\vp|^{p+1}
	\le (p+1) c_1 \io |\na\vp|^p |D^2\vp|+ c_1 \io |\na\vp|^{p+1}.
  \eas
  Since
  \bas
	\io |\na\vp|^p |D^2\vp|
	\le \bigg\{ \io |\na\vp|^{p-2} |D^2\vp|^2 \bigg\}^\frac{1}{2} \cdot \bigg\{ \io |\na\vp|^{p+2} \bigg\}^\frac{1}{2}
  \eas
  and
  \bas
	\io |\na\vp|^{p+1} 
	\le \bigg\{ \io |\na\vp|^p \bigg\}^\frac{1}{2} \cdot \bigg\{ \io |\na\vp|^{p+2} \bigg\}^\frac{1}{2}
  \eas
  by Young's inequality, this implies (\ref{72.1}).
\qed
We are now prepared to examine the functional in (\ref{04}) in order to rule out any finite-time blow-up in
(\ref{0}) involving bounded temperatures:\abs
\proofc of Proposition \ref{prop_ext}. \quad
  Let us assume on the contrary that in Proposition \ref{prop_loc} we have $\tm<\infty$, but that
  \be{e2}
	\Theta\le c_1
	\qquad \mbox{in } \Om\times (0,\tm)
  \ee
  with some $c_1>0$.
  Then the continuity of $\gamma$, $\gamma'$, $f$ and $f'$ on $[0,\infty)$ as well as the positivity of $\gamma$ ensure the existence
  of $c_i>0, i\in\{2,3,4,5,6\}$, such that
  \be{e22}
	c_2 \le \gamma(\Theta) \le c_3,
	\quad
	|\gamma'(\Theta)| \le c_4,
	\quad
	|f(\Theta)| \le c_5
	\quad \mbox{and} \quad
	|f'(\Theta)| \le c_6
	\qquad \mbox{in } \Om\times (0,\tm),
  \ee
  while Lemma \ref{lem11} yields 
  $\al_1\in (0,1)$ and $c_7>0$ fulfilling
  \be{e3}
	\|v(\cdot,t)\|_{C^{\al_1}(\bom)} \le c_7
	\qquad \mbox{for all } t\in (0,\tm).
  \ee
  To make suitable use of this, we pick any $p>n$ and utilize the first equation in (\ref{0v}) in verifying by straightforward 
  computation that
  \bea{e4}
	\frac{1}{p} \frac{d}{dt} \io |\na v|^p
	&=& \io |\na v|^{p-2} \na v \cdot \Big\{ \gamma(\Theta) \na\Del v + \gamma'(\Theta) \Del v \na\Theta \Big\} \nn\\
	& & - \io \Big\{ |\na v| ^{p-2} \Del v + \na |\na v|^{p-2} \cdot \na v\Big\} \cdot \gamma'(\Theta) \na\Theta\cdot\na v \nn\\
	& & + a\io |\na v|^p - a^2 \io |\na v|^{p-2} \na v\cdot\na u \nn\\
	& & - \io \Big\{ |\na v|^{p-2} \Del v + \na |\na v|^{p-2} \cdot\na v \Big\} f'(\Theta)\cdot\na\Theta \nn\\
	&=& - \io \gamma(\Theta) |\na v|^{p-2} |D^2 v|^2 
	- \frac{4(p-2)}{p^2} \io \gamma(\Theta) \Big| \na |\na v|^\frac{p}{2}\Big|^2 \nn\\
	& & - \frac{1}{2} \io \gamma'(\Theta) |\na v|^{p-2} \na\Theta\cdot\na |\na v|^2 
	- \io \gamma'(\Theta) (\na\Theta\cdot\na v) \na v\cdot\na |\na v|^{p-2} \nn\\
	& & + \frac{1}{2} \int_{\pO} \gamma(\Theta) |\na v|^{p-2} \frac{\pa |\na v|^2}{\pa\nu} \nn\\
	& & + a\io |\na v|^p - a^2 \io |\na v|^{p-2} \na v\cdot\na u \nn\\
	& & - \io \Big\{ |\na v|^{p-2} \Del v + \na |\na v|^{p-2} \cdot\na v \Big\} f'(\Theta)\cdot\na\Theta 
	\qquad \mbox{for all } t\in (0,\tm).
  \eea
  Here since $v_t=0$ and hence $\gamma(\Theta)\Del v = - \gamma'(\Theta) \na v\cdot\na\Theta + f'(\Theta)\cdot\na\Theta$
  on $\pO\times (0,\tm)$, using Lemma \ref{lem71} i) and (\ref{e22}) along with Young's inequality 
  we can find $c_8>0$ and $c_9>0$ such that
  \bea{e41}
	\frac{1}{2} \int_{\pO} \gamma(\Theta) |\na v|^{p-2} \frac{\pa |\na v|^2}{\pa\nu}
	&\le& 
	\frac{1}{2} \int_{\pO} \gamma(\Theta) |\na v|^{p-2} \cdot \Big\{ 2\frac{\pa v}{\pa\nu} \Del v + c_8 |\na v|^2 \Big\} \nn\\
	&=& - \int_{\pO} \gamma'(\Theta) |\na v|^{p-2} \frac{\pa v}{\pa\nu} \na v\cdot\na\Theta
	+ \int_{\pO} |\na v|^{p-2} \frac{\pa v}{\pa\nu} f'(\Theta)\cdot\na\Theta \nn\\
	& & + \frac{c_8}{2} \int_{\pO} \gamma(\Theta) |\na v|^p \nn\\ 
	&\le& c_9 \int_{\pO} |\na\Theta|\cdot |\na v|^p
	+ c_9 \int_{\pO} |\na\Theta|\cdot |\na v|^{p-1}
	+ c_9 \int_{\pO}|\na v|^p \nn\\
	&\le& c_9 \int_{\pO} |\na\Theta|\cdot |\na v|^p
	+ 2c_9 \int_{\pO} |\na v|^p
	+ c_9 \int_{\pO} |\na\Theta|^p
  \eea
  for all $t\in (0,\tm)$.
  Again by (\ref{e22}) and Young's inequality, we moreover find $c_i>0$, $i\in \{10,...,15\}$, such that
  \bea{e66}
	- \frac{1}{2} \io \gamma'(\Theta) |\na v|^{p-2} \na\Theta\cdot\na |\na v|^2 
	&=& - \io \gamma'(\Theta) |\na v|^{p-2} \na\Theta\cdot (D^2 v\cdot\na v) \nn\\	
	&\le& c_{10} \io |\na\Theta| \cdot |\na v|^{p-1} |D^2 v| \nn\\
	&\le& \frac{c_2}{4} \io |\na v|^{p-2} |D^2 v|^2 	
	+ c_{11} \io |\na\Theta|^2 |\na v|^p,
  \eea
  and that, similarly,
  \bea{e67}
	- \io \gamma'(\Theta) (\na\Theta\cdot\na v) \na v\cdot\na |\na v|^{p-2} 
	&\le& c_{12} \io |\na\Theta| \cdot |\na v|^{p-1} |D^2 v| \nn\\
	&\le& \frac{c_2}{4} \io |\na v|^{p-2} |D^2 v|^2
	+ c_{13} \io |\na \Theta|^2 |\na v|^p
  \eea
  as well as
  \bea{e68}
	& & \hs{-40mm}
	- \io \Big\{ |\na v|^{p-2} \Del v + \na |\na v|^{p-2} \cdot\na v \Big\} f'(\Theta)\cdot\na\Theta  \nn\\
	&\le& c_{14} \io |\na\Theta| \cdot |\na v|^{p-2} |D^2 v| \nn\\
	&\le& \frac{c_2}{4} \io |\na v|^{p-2} |D^2 v|^2
	+ c_{15} \io |\na\Theta|^2 |\na v|^{p-2} \nn\\
	&\le& \frac{c_2}{4} \io |\na v|^{p-2} |D^2 v|^2
	+ c_{15} \io |\na\Theta|^p 
	+ c_{15} \io |\na v|^p
  \eea
  for all $t\in (0,\tm)$.
  As furthermore, again by Young's inequality,
  \bas
	-a^2\io |\na v|^{p-2} \na v\cdot\na u
	\le a^2 \io |\na v|^p + a^2 \io |\na u|^p
	\qquad \mbox{for all } t\in (0,\tm),
  \eas
  upon inserting (\ref{e41}), (\ref{e66}), (\ref{e67}) and (\ref{e68}) into (\ref{e4}) we obtain $c_{16}>0$ such that
  for all $t\in (0,\tm)$,
  \bea{e8}
	\frac{d}{dt} \io |\na v|^p
	+ \frac{pc_2}{4} \io |\na v|^{p-2} |D^2 v|^2 
	&\le& c_{12} \io |\na \Theta|^2 |\na v|^p 
	+ c_{16} \int_{\pO} |\na\Theta| \cdot |\na v|^p \nn\\
	& & + c_{16} \io |\na v|^p 
	+ c_{16} \io |\na u|^p
	+ c_{16} \io |\na\Theta|^p \nn\\
	& & + c_{16} \int_{\pO} |\na v|^p
	+ c_{16} \int_{\pO} |\na \Theta|^p.
  \eea
  Here the crucial 
  first two summands on the right can be compensated by exploiting the diffusive action in the third equation from (\ref{0v}):
  A variant of the testing procedure performed above, namely, readily shows that thanks to Lemma \ref{lem71} ii) and
  (\ref{e22}) there exist $c_i>0$, $i\in \{17,18,19\}$, such that
  \bea{e9}
	\frac{1}{p} \frac{d}{dt} \io |\na\Theta|^p
	&=& \io |\na\Theta|^{p-2} \na\Theta \cdot \na \Big\{
	D\Del\Theta + \gamma(\Theta) |\na v-a\na u|^2
	+ f(\Theta) \cdot (\na v-a\na u)\Big\} \nn\\
	&=& - \frac{D}{2} \io \na |\na\Theta|^{p-2} \cdot\na |\na\Theta|^2
	+ \frac{D}{2} \int_{\pO} |\na\Theta|^{p-2} \frac{\pa |\na\Theta|^2}{\pa\nu}
	- D \io |\na\Theta|^{p-2} |D^2 \Theta|^2 \nn\\
	& & - \io \gamma(\Theta) \cdot \Big\{ |\na\Theta|^{p-2} \Del\Theta + \na |\na\Theta|^{p-2} \cdot\na\Theta\Big\}
		|\na v-a\na u|^2 \nn\\
	& & - \io \gamma(\Theta) \cdot \Big\{ |\na\Theta|^{p-2} \Del\Theta + \na |\na\Theta|^{p-2} \cdot\na\Theta\Big\}
		f(\Theta)\cdot (\na v-a\na u) \nn\\
	&\le& - \frac{D}{2} \io |\na \Theta|^{p-2} |D^2 \Theta|^2 
	+ c_{17} \int_{\pO} |\na\Theta|^p \nn\\
	& & + c_{17} \io |\na\Theta|^{p-2} |D^2\Theta| \cdot \big\{ |\na v|^2 + |\na u|^2 \big\} 
	+ c_{17} \io |\na\Theta|^{p-2} |D^2\Theta| \cdot \big\{ |\na v| + |\na u| \big\} \nn\\
	&\le& - \frac{D}{4} \io |\na \Theta|^{p-2} |D^2 \Theta|^2 
	+ c_{17} \int_{\pO} |\na\Theta|^p \nn\\
	& & + c_{18} \io |\na\Theta|^{p-2} |\na v|^4
	+ c_{18} \io |\na\Theta|^{p-2} |\na u|^4 \nn\\
	& & + c_{18} \io |\na\Theta|^{p-2} |\na v|^2
	+ c_{18} \io |\na\Theta|^{p-2} |\na u|^2 \nn\\
	&\le& - \frac{D}{4} \io |\na \Theta|^{p-2} |D^2 \Theta|^2 
	+ c_{18} \io |\na\Theta|^{p-2} |\na v|^4
	+ c_{18} \io |\na\Theta|^{p-2} |\na u|^4 \nn\\
	& & + c_{19} \io |\na v|^p
	+ c_{19} \io |\na u|^p
	+ c_{19} \io |\na\Theta|^p 
	+ c_{17} \int_{\pO} |\na\Theta|^p
  \eea
  for all $t\in (0,\tm)$.
  Apart from that, the second equation in (\ref{0v}) implies that due to Young's inequality,
  \bas
	\frac{1}{p+2} \frac{d}{dt} \io |\na u|^{p+2}
	= \io |\na u|^p \na u\cdot \na (v-au)
	\le \io |\na u|^p \na u \cdot \na v
	\le \io |\na u|^{p+2} + \io |\na v|^{p+2}
  \eas
  for all $t\in (0,\tm)$,
  whence again using Young's inequality to estimate $\io |\na u|^p \le \io |\na u|^{p+2} + |\Om|$
  for $t\in (0,\tm)$, from (\ref{e8}) and (\ref{e9}) we infer the existence of $c_{20}>0$ and $c_{21}>0$ such that
  \bas
	y(t):= 1 + \io |\na v(\cdot,t)|^p + \io |\na u(\cdot,t)|^{p+2} + \io |\na\Theta(\cdot,t)|^p,
	\qquad t\in [0,\tm),
  \eas
  satisfies
  \bea{e10}
	& & \hs{-20mm}
	y'(t) + c_{20} \io |\na v|^{p-2} |D^2 v|^2
	+ c_{20} \io |\na\Theta|^{p-2} |D^2\Theta|^2 \nn\\
	&\le& c_{21} y(t) 
	+ c_{21} \io |\na\Theta|^2 |\na v|^p
	+ c_{21} \io |\na\Theta|^{p-2} |\na v|^4
	+ c_{21} \io |\na\Theta|^{p-2} |\na u|^4 \nn\\
	& & + c_{21} \int_{\pO} |\na \Theta| \cdot |\na v|^p
	+ c_{21} \int_{\pO} |\na v|^p
	+ c_{21} \int_{\pO} |\na\Theta|^p
  \eea
  for all $t\in (0,\tm)$.
  We now rely on Lemma \ref{lem2} as well as Lemma \ref{lem72}
  and again (\ref{e2}) to pick $c_{22}>0$ and $c_{23}>0$ such that
  \be{e11}
	\io |\na\Theta|^{p+2} \le c_{22} \io |\na\Theta|^{p-2} |D^2\Theta|^2
	+ c_{22}
	\qquad \mbox{for all } t\in (0,\tm),
  \ee
  and that 
  \be{e111}
	\int_{\pO} |\na\vp|^{p+1} \le c_{23} \cdot \bigg\{ \io |\na\vp|^{p-2} |D^2 \vp|^2 + \io |\na \vp|^p \bigg\}^\frac{1}{2} \cdot
		\bigg\{ \io |\na\vp|^{p+2} \bigg\}^\frac{1}{2}
	\qquad \mbox{for all } \vp\in C^2(\bom),
  \ee
  and then employ Young's inequality four six times to find $c_{24}>0$ and $c_{25}>0$ fulfilling
  \bea{e12}
	& & \hs{-20mm}
	c_{21} \io |\na\Theta|^2 |\na v|^p
	+ c_{21} \io |\na\Theta|^{p-2} |\na v|^4
	+ c_{21} \io |\na\Theta|^{p-2} |\na u|^4 \nn\\
	&\le& \frac{c_{20}}{2c_{22}} \io |\na\Theta|^{p+2}
	+ c_{24} \io |\na v|^{p+2}
	+ c_{24} \io |\na u|^{p+2}
	\qquad \mbox{for all } t\in (0,\tm)
  \eea
  and 
  \bea{e122}
	& & \hs{-20mm}
	c_{21} \int_{\pO} |\na \Theta| \cdot |\na v|^p
	+ c_{21} \int_{\pO} |\na v|^p
	+ c_{21} \int_{\pO} |\na\Theta|^p \nn\\
	&\le& \frac{c_{20}}{2\sqrt{c_{22}} c_{23}} \int_{\pO} |\na\Theta|^{p+1}
	+ c_{25} \int_{\pO} |\na v|^{p+1}
	+ c_{25}
	\qquad \mbox{for all } t\in (0,\tm).
  \eea
  Here, (\ref{e11}) shows that on the right-hand side of (\ref{e12}) we have
  \bea{e123}
	\frac{c_{20}}{2c_{22}} \io |\na\Theta|^{p+2}
	\le \frac{c_{20}}{2} \io |\na\Theta|^{p-2} |D^2\Theta|^2
	+ \frac{c_{20}}{2} 
	\qquad \mbox{for all } t\in (0,\tm),
  \eea
  whereas a combination of (\ref{e111}) with Young's inequality and (\ref{e11}) estimates the first expression on the right
  of (\ref{e122}) according to
  \bea{e124}
	\frac{c_{20}}{2\sqrt{c_{22}} c_{23}} \int_{\pO} |\na\Theta|^{p+1} 
	&\le& \frac{c_{20}}{2\sqrt{c_{22}}} \cdot \bigg\{ \io |\na\Theta|^{p-2} |D^2\Theta|^2 + \io |\na\Theta|^p 
		\bigg\}^\frac{1}{2} \cdot \bigg\{ \io |\na\Theta|^{p+2} \bigg\}^\frac{1}{2} \nn\\
	&\le&
	\frac{c_{20}}{4} \io |\na\Theta|^{p-2} |D^2\Theta|^2
	+ \frac{c_{20}}{4} \io |\na\Theta|^p
	+ \frac{c_{20}}{4c_{22}} \io |\na\Theta|^{p+2} \nn\\
	&\le& \frac{c_{20}}{2} \io |\na\Theta|^{p-2} |D^2\Theta|^2
	+ \frac{c_{20}}{4} \io |\na\Theta|^p
	+ \frac{c_{20}}{4}
  \eea
  for all $t\in (0,\tm)$.
  Apart from that, an application of (\ref{e111}) to the first solution component in (\ref{0v}), followed by Young's inequality,
  confirms that on the right of (\ref{e122}) we moreover have
  \bas
	c_{25} \int_{\pO} |\na v|^{p+1}
	&\le& c_{23} c_{25} \cdot \bigg\{ \io |\na v|^{p-2} |D^2 v|^2 + \io |\na v|^p \bigg\}^\frac{1}{2} \cdot
		\bigg\{ \io |\na v|^{p+2} \bigg\}^\frac{1}{2} \nn\\
	&\le& \frac{c_{20}}{2} \io |\na v|^{p-2} |D^2 v|^2
	+ \frac{c_{20}}{2} \io |\na v|^p
	+ \frac{c_{23}^2 c_{25}^2}{2 c_{20}} \io |\na v|^{p+2}
	\qquad \mbox{for all } t\in (0,\tm).
  \eas
  In conjunction with (\ref{e12})-(\ref{e124}), this shows that (\ref{e10}) entails the existence of $c_{26}>0$ fulfilling
  \bas
	y'(t) + \frac{c_{20}}{2} \io |\na v|^{p-2} |D^2 v|^2
	\le c_{26} y(t) + c_{26} \io |\na v|^{p+2} + c_{26}
	\qquad \mbox{for all } t\in (0,\tm),
  \eas
  and in order to appropriately control the last integral herein, we rely on the equicontinuity property entailed by (\ref{e3})
  in applying Lemma \ref{lem_LW} along with 
  Lemma \ref{lem94} to choose $c_{27}>0$ in such a manner that
  \bas
	c_{27} \io |\na v|^{p+2} \le \frac{c_{20}}{2} \io |\na v|^{p-2} |D^2 v|^2 + c_{27}
	\qquad \mbox{for all } t\in (0,\tm).
  \eas
  Therefore,
  \bas
	y'(t) \le c_{26} y(t) + c_{26} + c_{27}
	\le (2 c_{26} + c_{27})y(t)
	\qquad \mbox{for all } t\in (0,\tm),
  \eas
  so that Gronwall's inequality ensures that $y$ is bounded in $(0,\tm)$, especially meaning that
  \bas
	\sup_{t\in (0,\tm)} \|\na v(\cdot,t)-a\na u(\cdot,t)\|_{L^p(\Om)} <\infty.
  \eas
  Together with Lemma \ref{lem94} and (\ref{e2}), in view of the identity $v-au=u_t$ and 
  the inequality $p>n$ this contradicts (\ref{Ext}) and hence proves the claim.
\qed
\mysection{Blow-up in the presence of superlinear $\gamma$}\label{sect4}
\subsection{A necessary condition on $(u_0,u_{0t},\Theta_0)$ for existence up to $t=T$}
A constitutive step toward our blow-up detection is now formed by the following rigorous version of the 
inequality announced near (\ref{001}), available under the assumption that in addition to (\ref{gf}), 
also (\ref{gp}), (\ref{gi}) and (\ref{gf1}) hold:
\begin{lem}\label{lem601}
  Let $a>0$ and $D>0$, suppose that $\gamma$ and $f$ comply with (\ref{gf}), (\ref{gp}), (\ref{gi}) and (\ref{gf1}),
  and assume (\ref{Init}). 
  Then
  \be{601.1}
	\int_0^T \io |\na u_t|^2 
	\le 2 \io \psi(\Theta_0)
	+ \frac{\Lam |\Om| T}{\gamma_0(T)}
	\qquad \mbox{for all } T\in (0,\tm),
  \ee
  where
  \be{psi}
	\psi(\xi):=\int_\xi^\infty \frac{d\sig}{\gamma(\sig)},
	\qquad \xi\ge 0
  \ee
  and 
  \be{g0}
	\gamma_0(T):=\gamma\Big( \inf_{(x,t)\in\Om\times (0,T)} \Theta(x,t) \Big),
	\qquad T\in (0,\tm).
  \ee
\end{lem}
\proof
  Using that $\psi'=-\frac{1}{\gamma}$, from the third equation in (\ref{0v}) we derive that
  \bas
	\frac{d}{dt} \io \psi(\Theta)
	&=& - \io \frac{1}{\gamma(\Theta)} \Theta_t  \\
	&=& - D \io \frac{1}{\gamma(\Theta)} \Del\Theta
	- \io |\na u_t|^2
	- \io \frac{1}{\gamma(\Theta)} f(\Theta)\cdot\na u_t
	\qquad \mbox{for all } t\in (0,\tm),
  \eas
  where since $\gamma'\ge 0$,
  \bas
	-D \io \frac{1}{\gamma(\Theta)} \Del\Theta
	= - D \io \frac{\gamma'(\Theta)}{\gamma^2(\Theta)} |\na\Theta|^2 \le 0
	\qquad \mbox{for all } t\in (0,\tm),
  \eas
  and where by Young's inequality and (\ref{gf1}), 
  \bas
	- \io \frac{1}{\gamma(\Theta)} f(\Theta)\cdot\na u_t
	&\le& \frac{1}{2} \io |\na u_t|^2
	+ \frac{1}{2} \io \frac{|f|^2(\Theta)}{\gamma^2(\Theta)} \\
	&\le& \frac{1}{2} \io |\na u_t|^2
	+ \frac{\Lam}{2} \io \frac{1}{\gamma(\Theta)} \\
	&\le& \frac{1}{2} \io |\na u_t|^2
	+ \frac{\Lam |\Om|}{2\gamma_0(T)}
	\qquad \mbox{for all $t\in (0,T)$ and $T\in (0,\tm)$,}
  \eas
  because $\gamma(\Theta)\ge \gamma_0(T)$ in $\Om\times (0,T)$ for all $T\in (0,\tm)$ by (\ref{g0}), and again 
  by monotonicity of $\gamma$.
  Therefore,
  \bas
	\frac{d}{dt} \io \psi(\Theta)
	+ \frac{1}{2} \io |\na u_t|^2
	\le \frac{\Lam |\Om|}{2\gamma_0(T)}
	\qquad \mbox{for all $t\in (0,T)$ and $T\in (0,\tm)$,}
  \eas
  which upon an integration implies that, indeed,
  \bas
	\frac{1}{2} \int_0^T \io |\na u_t|^2
	&\le& \io \psi(\Theta_0) - \io \psi\big(\Theta(\cdot,T)\big)
	+ \frac{\Lam |\Om| T}{2\gamma_0(T)} \\
	&\le& \io \psi(\Theta_0) 
	+ \frac{\Lam |\Om| T}{2\gamma_0(T)}
	\qquad \mbox{for all $t\in (0,T)$ and $T\in (0,\tm)$,}
  \eas
  as $\psi$ is nonnegative.
\qed
A relationship between the expression on the right of (\ref{601.1}) and an integral $\io |\na u|^2$ is readily established:
\begin{lem}\label{lem63}
  Let $a>0$ and $D>0$, and assume (\ref{gf}), (\ref{gp}), (\ref{gi}), (\ref{gf1}) and (\ref{Init}).
  Then
  \be{63.1}
	\io |\na u_0|^2
	\le 2 \io |\na u(\cdot,t)|^2 + 2t\int_0^t \io |\na u_t|^2
	\qquad \mbox{for all } t\in (0,\tm).
  \ee
\end{lem}
\proof
  We represent $\na u$ according to
  \bas
	\na u(x,t)=\na u_0(x) + \int_0^t \na u_t(x,s) ds
	\qquad \mbox{for } (x,t)\in\Om\times (0,\tm)
  \eas
  to see by means of the Cauchy-Schwarz inequality that
  \bas
	\Big| \na u(x,t)-\na u_0(x)\Big|^2
	= \bigg| \int_0^t \na u_t(x,s) ds \bigg|^2 
	\le t\int_0^t |\na u_t(x,s)|^2 ds
	\qquad \mbox{for all $x\in\Om$ and } t\in (0,\tm).
  \eas
  Hence,
  \bas
	\io |\na u(\cdot,t)-\na u_0|^2
	\le t \int_0^t \io |\na u_t|^2
	\qquad \mbox{for all } t\in (0,\tm),
  \eas
  so that (\ref{63.1}) results upon observing that by Young's inequality,
  \bas
	\io |\na u_0|^2
	&=& \io \Big| \na u(\cdot,t)-\big(\na u(\cdot,t)-\na u_0\big) \Big|^2 \\
	&\le& 2 \io |\na u(\cdot,t)|^2
	+ 2 \io |\na u(\cdot,t)-\na u_0|^2
  \eas
  for all $t\in (0,\tm)$.
\qed
On testing the first equation in (\ref{0v}) by $v$, however, the difference between $\na u_t$ and $\na u$
can be estimated by some constant expression:
\begin{lem}\label{lem61}
  Let $a>0$ and $D>0$, and assume (\ref{gf}), (\ref{gp}), (\ref{gi}), (\ref{gf1}) and (\ref{Init}).
  Then
  \be{61.1}
	\int_0^T \io |\na u_t + a\na u|^2 
	\le \frac{3a^2 e^{2aT}}{\gamma_0(T)} \cdot \io u_0^2
	+ \frac{2 e^{2aT}}{\gamma_0(T)} \cdot \io u_{0t}^2
	+ \frac{\Lam |\Om| e^{2aT}}{2a\gamma_0(T)} + \frac{\Lam |\Om| T}{\gamma_0(T)}
	\qquad \mbox{for all } T\in (0,\tm),
  \ee
  where $(\gamma_0(T))_{T\in (0,\tm)}$ is as in (\ref{g0}).
\end{lem}
\proof
  We multiply the first equation in (\ref{0v}) by $v$, integrate by parts and use Youngs's inequality along with (\ref{gf1})
  to see that
  \bas
	\frac{1}{2} \frac{d}{dt} \io v^2
	+ \io \gamma(\Theta) |\na v|^2
	&=& a \io v^2 - a^2 \io uv
	- \io f(\Theta)\cdot\na v \\
	&\le& a \io v^2 - a^2 \io uv
	+ \frac{1}{2} \io \gamma(\Theta) |\na v|^2
	+ \frac{1}{2} \io \frac{|f|^2(\Theta)}{\gamma(\Theta)} \\
	&\le& a \io v^2 - a^2 \io uv
	+ \frac{1}{2} \io \gamma(\Theta) |\na v|^2
	+ \frac{\Lam |\Om|}{2}	
	\qquad \mbox{for all } t\in (0,\tm),
  \eas  
  while testing the second equation in (\ref{0v}) against $u$ yields
  \bas
	\frac{1}{2} \frac{d}{dt} \io u^2 = \io uv - a \io u^2
	\qquad \mbox{for all } t\in (0,\tm).
  \eas
  By linear combination of these inequalities, we obtain that since $a\ge 0$,
  \bas
	y(t):=\frac{1}{2} \io v^2 + \frac{a^2}{2} \io u^2,
	\qquad t\in [0,\tm),
  \eas
  satisfies
  \bea{61.3}
	y'(t) + \frac{1}{2} \io \gamma(\Theta) |\na v|^2
	&\le& a \io v^2 - a^3 \io u^2 
	+ \frac{\Lam |\Om|}{2} 
	\le a \io v^2
	+ \frac{\Lam |\Om|}{2} \nn\\
	&\le& 2a y(t) + \frac{\Lam |\Om|}{2}
	\qquad \mbox{for all } t\in (0,\tm).
  \eea
  In particular, $y' \le 2ay+\frac{\Lam |\Om|}{2}$ on $(0,\tm)$, so that
  \bas
	y(t) 
	&\le& y(0) e^{2at} + \frac{\Lam |\Om|}{2} \int_0^t e^{2a(t-s)} ds \\
	&=& y(0) e^{2at} + \frac{\Lam |\Om|}{4a} \big( e^{2at} - 1 \big) \\
	&\le& \Big( y(0) + \frac{\Lam |\Om|}{4a} \Big) e^{2at}
	\qquad \mbox{for all } t\in (0,\tm),
  \eas
  whereupon an integration in (\ref{61.3}) shows that, by nonnegativity of $y$,
  \bea{61.4}
	\frac{1}{2} \int_0^T \io \gamma(\Theta) |\na v|^2
	&\le& y(0) + 2a\int_0^T y(t) dt + \frac{\Lam |\Om| T}{2} \nn\\
	&\le& y(0) + \Big(2ay(0)+ \frac{\Lam |\Om|}{2}\Big) \int_0^T e^{2at} dt + \frac{\Lam |\Om| T}{2} \nn\\
	&=& y(0) + \Big(2ay(0)+ \frac{\Lam |\Om|}{2}\Big) \cdot \frac{1}{2a} \big( e^{2at}-1\big) + \frac{\Lam |\Om| T}{2} \nn\\
	&\le& \Big(y(0) + \frac{\Lam |\Om|}{4a}\Big) e^{2at} + \frac{\Lam |\Om| T}{2}
	\qquad \mbox{for all } t\in (0,\tm).
  \eea
  Since Young's inequality warrants that 
  \bas
	y(0)
	= \frac{1}{2} \io (u_{0t} + au_0)^2 + \frac{a^2}{2} \io u_0^2
	\le \io u_{0t}^2 + \frac{3a^2}{2} \io u_0^2
  \eas
  using the monotonicity of $\gamma$ in estimating
  \bas
	\frac{1}{2} \int_0^T \io \gamma(\Theta) |\na v|^2
	\ge \frac{\gamma_0(T)}{2} \int_0^T \io |\na v|^2
	= \frac{\gamma_0(T)}{2} \int_0^T \io |\na u_t + a\na u|^2
	\qquad \mbox{for } T\in (0,\tm)
  \eas
  turns (\ref{61.4}) into (\ref{61.1}).
\qed
When combined with Lemma \ref{lem601}, this immediately entails an upper bound also for $\int_0^T \io|\na u|^2$:
\begin{lem}\label{lem62}
  If $a>0$ and $D>0$, and if (\ref{gf}), (\ref{gp}), (\ref{gi}), (\ref{gf1}) and (\ref{Init}) hold, then
  with $\psi$ and $(\gamma_0(T))_{T\in (0,\tm)}$ as in (\ref{psi}) and (\ref{g0}),
  \be{62.1}
	\int_0^T \io |\na u|^2 
	\le \frac{4}{a^2} \io \psi(\Theta_0) 
	+ \frac{6 e^{2aT}}{\gamma_0(T)} \cdot \io u_0^2
	+ \frac{4 e^{2aT}}{a^2\gamma_0(T)} \cdot \io u_{0t}^2
	+ \frac{\Lam |\Om| e^{2aT}}{a^3\gamma_0(T)}
	+ \frac{4\Lam |\Om| T}{a^2 \gamma_0(T)}
  \ee
  for all $T\in (0,\tm)$.
\end{lem}
\proof
  Since Young's inequality implies that
  for all $T\in (0,\tm)$ we have
  \bas
	\int_0^T \io |\na u|^2
	= \frac{1}{a^2} \int_0^T \io \Big| \na u_t - (\na u_t + a\na u)\Big|^2
	\le \frac{2}{a^2} \int_0^T \io |\na u_t|^2
	+ \frac{2}{a^2} \int_0^T \io |\na u_t + a\na u|^2,
  \eas
  this is a direct consequence of Lemma \ref{lem61} when combined with (\ref{601.1}).
\qed
To develop this further, we now integrate the inequality from Lemma \ref{lem63} and once more explicitly use
Lemma \ref{lem601}. We thereby obtain the following statement which can be viewed as providing, given $T>0$, 
a condition on the initial data that is necessary for a solution to exist up to time $T$.
\begin{lem}\label{lem64}
  Let $a>0$ and $D>0$, suppose that (\ref{gf}), (\ref{gp}), (\ref{gi}), (\ref{gf1}) and (\ref{Init}) are satisfied, and
  let $\psi$ and $(\gamma_0(T))_{T\in (0,\tm)}$ be as in (\ref{psi}) and (\ref{g0}).
  Then for all $T\in (0,\tm)$,
  \bea{64.1}
	\io |\na u_0|^2
	&\le& \Big\{ \frac{8}{a^2 T} + 4T\Big\} \cdot \io \psi(\Theta_0)
	+ \frac{12 e^{2aT}}{T \gamma_0(T)} \cdot \io u_0^2
	+ \frac{8 e^{2aT}}{a^2 T \gamma_0(T)} \cdot \io u_{0t}^2 \nn\\
	& & 
	+ \frac{2\Lam |\Om| e^{2aT}}{a^3 T\gamma_0(T)}
	+ \frac{8 \Lam |\Om|}{a^2 \gamma_0(T)}
	+ \frac{2\Lam |\Om| T^2}{\gamma_0(T)}.
  \eea
\end{lem}
\proof
  For fixed $T\in (0,\tm)$, from (\ref{63.1}) we obtain that
  \bas
	\frac{1}{T} \io |\na u_0|^2
	\le \frac{2}{T} \io |\na u(\cdot,t)|^2
	+ 2 \int_0^T \io |\na u_t|^2
	\qquad \mbox{for all } t\in (0,T),
  \eas
  which when integrated over $(0,T)$ implies that
  \bas
	\io |\na u_0|^2
	\le \frac{2}{T} \int_0^T \io |\na u|^2
	+ 2T \int_0^T \io |\na u_t|^2.
  \eas
  Since Lemma \ref{lem62} in conjunction with Lemma \ref{lem601} reveals that
  \bas
	& & \hs{-16mm}
	\frac{2}{T} \int_0^T \io |\na u|^2
	+ 2T \int_0^T \io |\na u_t|^2 \\
	&\le& \frac{2}{T} \cdot \bigg\{
	\frac{4}{a^2} \io \psi(\Theta_0) 
	+ \frac{6 e^{2aT}}{\gamma_0(T)} \cdot \io u_0^2
	+ \frac{4 e^{2aT}}{a^2\gamma_0(T)} \cdot \io u_{0t}^2
	+ \frac{\Lam |\Om| e^{2aT}}{a^3\gamma_0(T)}
	+ \frac{4\Lam |\Om| T}{a^2 \gamma_0(T)}
	\bigg\} \nn\\
	& & + 2T\cdot\bigg\{
	2 \io \psi(\Theta_0)
	+ \frac{\Lam |\Om| T}{\gamma_0(T)} \bigg\},
  \eas
  upon rearrangement this establishes (\ref{64.1}).
\qed
\subsection{Conclusion. Proofs of Theorem \ref{theo65} and Theorem \ref{theo66}}
The first of the announced main results on finite-time blow-up in (\ref{0}) can now be achieved by simply making sure
that (\ref{64.1}) must be violated if (\ref{65.1}) is satisfied with some appropriately large $C$:\abs
\proofc of Theorem \ref{theo65}. \quad
  Given $a>0$ and functions $\gamma$ and $f$ fulfilling (\ref{gf}), (\ref{gp}), (\ref{gi}) and (\ref{gf1})
  with some $\Lam>0$,
  we take $\psi$ from (\ref{psi}), and for fixed $T>0$ 
  we choose any $C=C(a,\gamma,f,T)>0$ large enough fulfilling
  \be{65.34}
	C>\frac{12 e^{2aT}}{T\gamma(0)}
  \ee
  and
  \be{65.35}
	C>\frac{8 e^{2aT}}{a^2 T\gamma(0)}
  \ee
  as well as
  \be{65.36}
	C> \Big\{ \frac{8}{a^2 T} + 4T\Big\} \cdot  |\Om| \psi(0) 
	+ \frac{2\Lam |\Om| e^{2aT}}{a^3 T\gamma(0)}
	+ \frac{8 \Lam |\Om|}{a^2\gamma(0)}
	+ \frac{2\Lam |\Om| T^2}{\gamma(0)}.
  \ee
  Then assuming that $D>0$, and that (\ref{Init}) and (\ref{65.1}) hold, we first observe that by 
  monotonicity of $\gamma$, the number defined in (\ref{g0}) satisfies $\gamma_0(T)\ge \gamma(0)$, 
  and apart from that we note that since the function $\psi$ from (\ref{psi}) satisfies $\psi(\xi)\le \psi(0)$ 
  for all $\xi\ge 0$, we can estimate
  \bas
	\io \psi(\Theta_0) \le |\Om| \psi(0).
  \eas
  From (\ref{65.34})-(\ref{65.36}) we therefore obtain that 
  \bea{65.6}
	\io |\na u_0|^2
	&>& \frac{12 e^{2aT}}{T\gamma(0)} \io u_0^2
	+ \frac{8 e^{2aT}}{a^2 T\gamma(0)} \io u_{0t}^2 \nn\\
	& & +\Big\{ \frac{8}{a^2 T} + 4T \Big\} \cdot |\Om| \psi(0) 
	+ \frac{2\Lam |\Om| e^{2aT}}{a^3 T\gamma(0)}
	+ \frac{8 \Lam |\Om|}{a^2 \gamma(0)}
	+ \frac{2\Lam |\Om| T^2}{\gamma(0)} \nn\\
	&\ge& \frac{12 e^{2aT}}{T\gamma_0(T)} \io u_0^2
	+ \frac{8 e^{2aT}}{a^2 T\gamma_0(T)} \io u_{0t}^2 \nn\\
	& & + \Big\{\frac{8}{a^2 T} + 4T \Big\} \cdot \io \psi(\Theta_0)
	+ \frac{2\Lam |\Om| e^{2aT}}{a^3 T\gamma_0(T)}
	+ \frac{8 \Lam |\Om|}{a^2 \gamma_0(T)}
	+ \frac{2\Lam |\Om| T^2}{\gamma_0(T)}.
  \eea
  We therefore indeed must have $\tm\le T$, for otherwise we could employ Lemma \ref{lem64} to see that (\ref{65.6}) cannot hold.
\qed
To see that also the assumptions of Theorem \ref{theo66} enforce blow-up by a given time $T>0$ of the number $C$ in
(\ref{66.3}) is sufficiently large, let us briefly record the outcome of a simple comparison argument, once again explicitly
drawing on the nonnegativity of the heat source $\gamma(\Theta)|\na u_t|^2$: 
\begin{lem}\label{lem99}
  If $a>0$ and $D>0$, and if (\ref{gf}), (\ref{gp}), (\ref{gi}), (\ref{gf1}) and (\ref{Init}) hold, then
  \be{99.1}
	\Theta(x,t) \ge \inf_\Om \Theta_0 - \frac{\Lam t}{4}
	\qquad \mbox{for all $x\in\Om$ and } t\in (0,\tm).
  \ee
\end{lem}
\proof
  From (\ref{0}) and Young's inequality it follows that thanks to (\ref{gf1}),
  \bas
	\Theta_t 
	&=& D\Del\Theta + \gamma(\Theta) |\na u_t|^2
	+ f(\Theta)\cdot\na u_t \\
	&\ge& D\Del\Theta - \frac{|f|^2(\Theta)}{4\gamma(\Theta)} \\
	&\ge& D\Del\Theta - \frac{\Lam}{4}
	\qquad \mbox{in } \Om\times (0,\tm).
  \eas
  On the other hand, for
  \bas
	\un{\Theta}(x,t):=\Theta_\star - \frac{\Lam t}{4},
	\qquad x\in\bom, \ t\ge 0,
  \eas
  with $\Theta_\star:=\inf_\Om \Theta_0$, we have
  \bas
	\un{\Theta}_t - D\Del\un{\Theta} + \frac{\Lam}{4}
	= - \frac{\Lam}{4} + \frac{\Lam}{4}=0
	\qquad \mbox{in } \Om\times (0,\infty),
  \eas
  so that since $\un{\Theta}(x,0)=\Theta_\star \le \Theta(x,0)$ for all $x\in\Om$, and since 
  clearly $\frac{\pa \un{\Theta}}{\pa\nu}=0$ on $\pO\times (0,\infty)$,
  a comparison principle asserts that $\Theta\ge\un{\Theta}$ in $\Om\times (0,\tm)$, which is equivalent to (\ref{99.1}).
\qed
Indeed, the previous lemma along with the monotonicity of $\gamma$ implies that assuming (\ref{66.3}) 
with suitably large $C$ ensures smallness not only of the integral $\io \psi(\Theta_0)$ but also of the factor
$\frac{1}{\gamma_0(T)}$ appearing in all further summands in (\ref{64.1}). 
Therefore, also our result on blow-up driven by large initial temperature distributions is a consequence of Lemma \ref{lem64}:\abs
\proofc of Theorem \ref{theo66}. \quad
  We note that (\ref{gi}) implies that for $\psi$ as in (\ref{psi}) we have $\psi(\xi)\to 0$ as $\xi\to\infty$, and that
  together with (\ref{gp}) this means that necessarily $\gamma(\xi)\to+\infty$ as $\xi\to\infty$.
  Given $\eta>0$ and $M>0$ we can therefore choose $C=C(a,\gamma,f,T,\eta,M)>0$ suitably large such that besides the inequality
  \be{66.4}
	C \ge \frac{\Lam T}{2},
  \ee
  we also have
  \be{66.5}
	c_1 \cdot |\Om| \psi(C) \le \frac{\eta}{4}
  \ee
  and
  \be{66.6}
	\frac{c_2 M}{\gamma(\frac{C}{2})} 
	+ \frac{c_3 M}{\gamma(\frac{C}{2})} 
	+ \frac{c_4}{\gamma(\frac{C}{2})} 
	\le \frac{\eta}{4},
  \ee
  where we have set
  \be{66.7}
	c_1\equiv c_1(a,T):= \frac{8}{a^2 T} + 4T
  \ee
  and
  \be{66.8}
	c_2\equiv c_2(a,T):= \frac{12 e^{2aT}}{T} 
  \ee
  and
  \be{66.9}
	c_3\equiv c_3(a,T):= \frac{8 e^{2aT}}{a^2 T}
  \ee
  as well as
  \be{66.10}
	c_4\equiv c_4(a,\gamma,f,T)
	:= \frac{2\Lam |\Om| e^{2aT}}{a^3}
	+ \frac{8 \Lam |\Om|}{a^2}
	+ 2\Lam |\Om| T^2.
  \ee
  Now if $D>0$, and if $(u_0,u_{0t},\Theta_0)$ satisfies (\ref{Init}), (\ref{66.1}), (\ref{66.2}) and (\ref{66.3}),
  assuming the corresponding solution to have the property that $\tm>T$
  we could combine Lemma \ref{lem99} with (\ref{66.3}) and (\ref{66.4}) to estimate
  \bas
	\Theta(x,t)
	\ge C - \frac{\Lam T}{4}
	\ge \frac{C}{2}
	\qquad \mbox{for all $x\in\Om$ and } t\in (0,T),
  \eas
  by monotonicity of $\gamma$ meaning that the quantity from (\ref{g0}) would satisfy
  \be{66.11}
	\gamma_0(T)\ge \gamma\Big(\frac{C}{2}\Big).
  \ee
  As (\ref{66.3}) furthermore ensures that, by an evident monotonicity property of $\psi$,
  \bas
	\io \psi(\Theta_0) \le |\Om| \psi(C),
  \eas
  in line with (\ref{66.7})-(\ref{66.10}) an application of Lemma \ref{lem64} would thus show that
  due to (\ref{66.2}),
  \bas
	\io |\na u_0|^2
	&\le& c_1 \io \psi(\Theta_0)
	+ \frac{c_2}{\gamma_0(T)} \io u_0^2
	+ \frac{c_3}{\gamma_0(T)} \cdot \io u_{0t}^2 
	+ \frac{c_4}{\gamma_0(T)} \\
	&\le& c_1 |\Om| \psi(C)
	+ \frac{c_2 M}{\gamma(\frac{C}{2})}
	+ \frac{c_3 M}{\gamma(\frac{C}{2})}
	+ \frac{c_4}{\gamma(\frac{C}{2})}.
  \eas
  But in conjunction with (\ref{66.5}) and (\ref{66.6}) this would imply that
  \bas
	\io |\na u_0|^2
	\le \frac{\eta}{4} + \frac{\eta}{4} <\eta,
  \eas
  contrary to our assumption in (\ref{66.1}).
  Accordingly, our hypothesis that $\tm\ge T$ must have been false, so that the claim becomes
  a consequence of Proposition \ref{prop_loc}.
\qed

\bigskip

{\bf Acknowlegements.} \quad
The author acknowledges support of the Deutsche Forschungsgemeinschaft (Project No. 444955436).\abs
{\bf Conflict of interest statement.} \quad
The author declares that he has no conflict of interest.\abs 
{\bf Data availability statement.} \quad
Data sharing is not applicable to this article as no datasets were
generated or analyzed during the current study.\abs

\end{document}